\documentclass{elsarticle}

\usepackage{graphicx}
\usepackage[utf8]{inputenc}
\usepackage{subfigure}
\usepackage{amsmath}
\usepackage{amsthm}
\usepackage{amssymb, mathtools, hyperref}
\usepackage{graphicx}
\usepackage{cleveref}
\usepackage{tikz}
\usetikzlibrary{patterns,shapes}
\usepackage{gastex}
\usepackage{todonotes}
\usepackage{cite}

\def\N{\mathbb N}
\def\A{\mathcal A}

\def\LL{\mathcal L}

\def\uu{\mathbf{u}}
\def\vv{\mathbf{v}}

\def\yy{\mathbf{y}}
\def\ff{\mathbf{f}}

\def\dd{\mathbf{d}}

\theoremstyle{definition}
\newtheorem{definition}{Definition}
\newtheorem{corollary}[definition]{Corollary}

\newtheorem{example}[definition]{Example}

\theoremstyle{plain}
\newtheorem{theorem}[definition]{Theorem}
\newtheorem{proposition}[definition]{Proposition}
\newtheorem{lemma}[definition]{Lemma}
\newtheorem{observation}[definition]{Observation}

\begin{document}
\begin{frontmatter}
\title{An upper bound on asymptotic repetitive threshold of balanced sequences via colouring of the Fibonacci sequence}

\author[cvut]{L\!'ubom\'ira Dvo\v r\'akov\'a}
\ead{lubomira.dvorakova@fjfi.cvut.cz}
\author[cvut]{Edita Pelantov\'a}
\ead{edita.pelantova@fjfi.cvut.cz}
\address[cvut]{ Czech Technical University in Prague, Czech Republic}

\begin{abstract}
We colour the Fibonacci sequence by suitable constant gap sequences to provide an upper bound on the asymptotic repetitive threshold of $d$-ary balanced sequences. The bound is attained for $d=2, 4$ and $8$ and we conjecture that it happens for infinitely many  even $d$'s. 
  
Our bound reveals an essential difference in behavior of the repetitive threshold and  the asymptotic repetitive threshold of balanced sequences. 
The repetitive threshold of $d$-ary  balanced sequences is known to be  at least  $1+\frac{1}{d-2}$ for each $d \geq 3$. In contrast, our bound implies that the asymptotic repetitive threshold of $d$-ary balanced sequences is at most $1+\frac{\tau^3}{2^{d-3}}$ for each $d\geq 2$, where $\tau$ is the golden mean.  
\end{abstract}

\end{frontmatter}

\section{Introduction}  The Fibonacci and the Thue-Morse sequences are the most prominent sequences of combinatorics on words. They often exhibit an extremal behaviour with respect to a studied property.  The property we focus on  in this note is repetition of motives in infinite strings.     
 
The $n$-th power of a non-empty word $u$ is a repetition of $n$ copies of $u$ if $n$ is a positive integer. 
In other words, $u^n$ denotes the prefix of length $|u|n$ of the infinite periodic sequence $uuu\cdots = u^\omega$, where $|u|$ is the length of $u$. Besides integer powers we can define also rational powers: any prefix of length $k$ of $u^\omega$ can be written as $u^e$, where $e=\frac{k}{|u|}$. 
For instance, a Czech word $kabelka$ (handbag) can be written in this formalism  as  $(kabel)^{7/5}$.

The {\em critical exponent}  $E(\uu )$ of an infinite sequence $\uu$ is defined as
$$E(\uu) =\sup\{e \in \mathbb{Q}: \  u ^e \  \text{is a factor of   } \uu  \  \text{for a non-empty word} \  u\}\,.
$$
In 1912,  Axel Thue proved that among all binary sequences the sequence generated by infinite iteration of the rewriting rule:  $\tt a \mapsto \tt ab$ and $\tt b \mapsto \tt ba$, i.e.,  the infinite sequence having prefixes   
$$\tt a \mapsto \tt ab \mapsto \tt abba \mapsto  \tt abbabaab \mapsto \tt abbabaabbaababba \mapsto \cdots\,,
$$
has the smallest possible critical exponent equal to 2. The sequence is now called the Thue-Morse sequence.  The famous Dejean's theorem (stated by Dejean~\citep{Dej72} and proved step by step by many people \citep{Pan84c, Mou92, MoCu07, Car07, CuRa11, Rao11}) says that the critical exponent of any $d$-ary sequence $\uu$ is greater than or equal to $1+\frac{1}{d-1}$ and this bound is for $d=2$ and $d\geq 5$ the best possible.  If we use the notation from~\citep{CuRa11}, we can write $RT(d)= 1+\frac{1}{d-1}$ for $d=2$ and $d\geq 5$, where $RT(d)$ is called the {\em repetitive threshold} and
$$ RT(d) = \inf\{ E(\uu) : \uu \text{\ is a sequence over a $d$-ary alphabet} \}\,.
$$

\medskip
Here we focus on balanced sequences. Let us recall that a sequence over a finite alphabet is {\em balanced} if, for any two of
its factors $u$ and $v$ of the same length, the number of occurrences of each letter
in $u$ and $v$ differs by at most one. Binary balanced aperiodic sequences were introduced in ~\citep{MoHe40}, they are called Sturmian. The Fibonacci sequence generated by the rewriting rule
$\tt a\mapsto \tt ab$ and $\tt b\mapsto \tt a$, i.e., having  prefixes 
$$\tt a\mapsto \tt ab \mapsto \tt aba\mapsto \tt abaab \mapsto \tt abaababa\mapsto \tt abaababaabaab\mapsto \cdots,$$
is a balanced sequence with the smallest critical exponent among all binary balanced sequences \citep{CaLu2000}, its value is $E(\ff) = 2+\frac{1}{\tau}$, where $\tau = \frac12(1+\sqrt{5})$ denotes the golden mean.  Hubert \citep{Hubert00} showed that any balanced sequence over any alphabet  arises by the so-called colouring of a Sturmian sequence by two constant gap sequences. A method to compute the critical exponent of  a balanced sequence is provided in \citep{DDP21}. 
The {\em repetitive threshold of balanced sequences}, i.e.,
$$RTB(d) = \inf\{E(\vv): \vv \text{ is a balanced sequence over a } d \text{-ary alphabet}\}\,, $$
was studied by several  authors. Today,  the threshold value $RTB(d)$ is known for all even $d$ and for $d\leq 11$, see~\citep{RSV19, Bar20, BaSh19,  DvOpPeSh2022}. It is worth  mentioning that for $d=4$ (see \citep{RSV19}) and all even $d\geq  14$ (see \citep{DvOpPeSh2022}), the threshold is attained on $d$-ary balanced sequences which arise by colouring of morphic images of the Fibonacci sequence.         
\medskip

In this  paper we concentrate  on  the {\em asymptotic critical exponent} $E^*(\uu)$  which 
 is defined to be $+\infty$,  if $E(\uu) = +\infty$ and
$$E^*(\uu) =\limsup_{n\to \infty}\{e \in \mathbb{Q}: \  u ^e \  \text{is a factor of  } \uu  \  \text{for some }  u \ \text{of length} \  n  \}\,,$$
 otherwise. If we define the {\em asymptotic repetitive threshold} as $$RT^*(d) =\inf\{E^*(\uu): \uu \text{\ is a sequence over a $d$-ary alphabet}\}\,, $$
we have $RT^*(d) =1 \ \text{for all $d\geq 2$}\,,$ see~\citep{DOP2022}.
The situation is more interesting if we  restrict our consideration to the {\em asymptotic repetitive threshold of balanced sequences}, i.e.,
$$RTB^*(d) =\inf\{E^*(\vv): \vv \text{\ is a  balanced sequence over a $d$-ary alphabet}\}\,. $$
 This threshold is bounded from below by $1+\frac{1}{2^{d-2}}$ and the  precise values of  $RTB^*(d)$ are known   for $d\leq 10$,  see ~\citep{DOP2022}.

 \medskip
 
In this note, we colour the Fibonacci sequence  by suitable constant gap sequences to provide the upper bound  $RTB^*(d) < 1+ \frac{\tau^3}{2^{d-3}}$ for $d \geq 2$. This bound is  refined for even $d$'s. The refined  bound equals $RTB^*(d)$  for $d=2, 4$ and $8$. We conjecture that this bound is attained also for infinitely many even values $d$.

\section{Preliminaries}
\label{Section_Preliminaries}
An \textit{alphabet} $\A$ is a finite set of symbols called \textit{letters}.
A \textit{word} over $\A$ of \textit{length} $n$ is a string $u = u_0 u_1 \cdots u_{n-1}$, where $u_i \in \A$ for all $i \in \{0,1, \ldots, n-1\}$.
The length of $u$ is denoted by $|u|$. 
The set of all finite words over $\A$ together with the operation of concatenation forms a monoid, denoted $\A^*$.
Its neutral element is the \textit{empty word} $\varepsilon$.
If $u = xyz$ for some $x,y,z \in \A^*$, then $x$ is a \textit{prefix} of $u$, $z$ is a \textit{suffix} of $u$ and $y$ is a \textit{factor} of $u$.
To any word $u$ over $\A$ with cardinality $\#\A = d$,  we assign its \textit{Parikh vector} $\Psi(u) \in \N^{d}$ defined as $(\Psi(u))_a = |u|_a$ for $a \in \A$, where $|u|_a$ is the number of letters $a$ occurring in $u$.

A \textit{sequence} over $\A$ is an infinite string $\uu = u_0 u_1 u_2 \cdots$, where $u_i \in \A$ for all $i \in \N$. We always denote sequences by bold letters. 

A sequence $\uu$ is \textit{eventually periodic} if $\uu = wvvv \cdots = wv^\omega$ for some $w, v \in \A^*$ and $v \neq \varepsilon.$  It is \textit{periodic} if $w=\varepsilon$. In both cases, the number $|v|$ is a \textit{period} of $\uu$. 
If $\uu$ is not eventually periodic, then it is \textit{aperiodic}.
A \textit{factor} of $\uu = u_0 u_1 u_2 \cdots$ is a word $y$ such that $y = u_i u_{i+1} u_{i+2} \cdots u_{j-1}$ for some $i, j \in \N$, $i \leq j$. 
The number $i$ is called an \textit{occurrence} of the factor $y$ in $\uu$.
In particular, if $i = j$, the factor $y$ is the empty word $\varepsilon$ and any index $i$ is its occurrence.
If $i=0$, the factor $y$ is a \textit{prefix} of $\uu$. 
The \textit{language} $\mathcal{L}(\uu)$ of the sequence $\uu$ is the set of all its factors.

If each factor of $\uu$ has infinitely many occurrences in $\uu$, the sequence $\uu$ is \textit{recurrent}.   If  $i < j$ are two consecutive occurrences of $w$ in $\uu$, then  the word $u_i u_{i+1} \cdots u_{j-1}$ is a \textit{return word} to $w$ in $\uu$.  A recurrent sequence $\uu$ said to be \textit{uniformly recurrent} if each factor $w \in \mathcal{L}(\uu)$ has only  finitely many return words. 

Let  $ \{r_1, r_2, \ldots, r_k\}$  be the set of return words to a prefix $w$ of $\uu$. Then $\uu$ can be written as a concatenation $\uu = r_{d_0}r_{d_1}r_{d_2} \cdots$ of return words to $w$. The sequence $\dd_\uu(w) = d_0d_1d_2 \cdots$ over the alphabet of cardinality $k$ is called the  \textit{derived sequence} of $\uu$ to $w$.
The concept of derived sequences was introduced by Durand~\citep{Dur98}.




A~sequence $\uu\in \A^{\mathbb{N}}$ is \textit{balanced} if for every letter $a \in \A$ and every pair of factors $u,v \in {\mathcal L}(\uu)$ with $|u|=|v|$, we have $|u|_a-|v|_a\leq 1$.  Every recurrent balanced sequence over an alphabet of any size is uniformly recurrent (see \citep{DDP21}).  Aperiodic balanced sequences over a~binary alphabet are called \textit{Sturmian sequences}. 


A \textit{morphism} over $\A$ is a mapping $\psi: \A^* \to \A^*$ such that $\psi(uv) = \psi(u)\psi(v)$  for all $u, v \in \A^*$. Morphisms can be naturally extended to $\A^{\mathbb{N}}$ by setting
$\psi(u_0 u_1 u_2 \cdots) = \psi(u_0) \psi(u_1) \psi(u_2) \cdots\,$. A \textit{fixed point} of a morphism $\psi$ is a sequence $\uu$ such that $\psi(\uu) = \uu$.

The asymptotic critical exponent of a uniformly recurrent sequence can be computed from its bispecial factors and their return words. Let us recall that a factor $w$ of $\uu$ is \textit{bispecial} if $wa, wb\in \mathcal{L}(\uu)$ for at least two distinct letters $a,b \in \A$ and $cw, dw\in \mathcal{L}(\uu)$ for at least two distinct letters $c,d \in \A$.

\begin{theorem}[{\citep{DDP21}}]
\label{thm:FormulaForCR}
Let $\uu$ be a uniformly recurrent aperiodic sequence.
Let $(w_n)$ be a sequence of all bispecial factors of $\uu$, ordered by their length.
For every $n \in \N$, let $v_n$ be a shortest return word to $w_n$ in $\uu$.
Then
$$
E^*(\uu) = 1 + \limsup_{n \to \infty} \left\{ \frac{|w_n|}{|v_n|} \right\}.
$$
\end{theorem}

\subsection{Fibonacci numbers}\label{sec:FibonacciNumbers}
\begin{definition}
Let $F_0=0, \ F_1=1$ and $F_{n+2}=F_{n+1}+F_{n}$ for each $n \in \mathbb N$. Then $F_n$ is called the {\em $n$-th Fibonacci number}.
\end{definition}

We list some properties of Fibonacci numbers we  exploit in proofs. 
Here $\tau$ denotes the golden mean, i.e., $\tau = \frac12 (\sqrt{5}+1)\doteq 1.618$  is the larger root of $x^2=x+1$.

\begin{enumerate}
    \item $F_{n+1}F_{n-1}-F_n^2=(-1)^n$ for each $n \in \mathbb N$;
    \item $\gcd(F_n, F_{n+1})=1$ for each $n \in \mathbb N$;
    \item $F_{n+1}-\tau F_n=\frac{(-1)^n}{\tau^n}$ for each $n \in \mathbb N$;
    \item $\lim_{n \to \infty}\frac{F_{n+1}}{F_n}=\tau$;
    \item $F_{n+2}-\tau F_{n+1}$ and $F_{n+1}-\tau F_n$ have distinct signs for each $n \in \mathbb N$ and $(|F_{n+1}-\tau F_n|)_{n=0}^{\infty}$ is a strictly decreasing sequence;
    \item $F_{m+1}F_{n+1}+F_m F_n=F_{m+n+1}$ for all $m,n \in \mathbb N$.
\end{enumerate}
The first, third and sixth statement may be proved by induction. The second statement follows from the first one. The fourth and fifth statement follows from the third one.

\subsection{Fibonacci sequence}
The Fibonacci sequence $\ff$ over $\{\tt a, \tt b\}$ is a fixed point of the morphism $\varphi$ defined by
$$\varphi(\tt a)=\tt ab, \ \varphi(\tt b)=\tt a\,,$$
i.e., 
$$\ff=\tt abaababaabaababaababaabaababaab\cdots$$
The Fibonacci sequence is the most prominent  Sturmian sequence.  The bispecial factors are exactly palindromic prefixes of $\ff$, see~\citep{DrJuPi2001}.  A result of   Vuillon~\citep{Vui01} implies  that each factor of $\ff$  has exactly two return words.  \begin{example} The occurrences of 
the palindromic prefix   $w=\tt aba$  in $\ff$ are $0, 3, 5, 8,11, \ldots$
Hence $f_0f_1f_2 = \tt aba$  and $f_3f_4 = \tt ab$
 are the two return words to $\tt aba$. Moreover,  
the factor $w=\tt aba$ is bispecial since  ${\tt a}w{\tt a}=\tt aabaa$,  ${\tt a}w{\tt b}=\tt aabab$ and  ${\tt b}w{\tt a}=\tt babaa$ occur in $\ff$.        
\end{example}

For our further study, we need to know the Parikh vectors of bispecial factors and their return words. Let us order bispecial factors of $\ff$ by length and denote $b_0=\varepsilon, b_1={\tt a}, b_2={\tt aba},\dots$.  As $b_n$ is a prefix of $\ff$, one of its return words is a~prefix of $\ff$, too.   
\begin{proposition}[{\citep{DvMePe20}}]\label{prop:returnWords}
Let $b_N$ be the $N$-th bispecial factor of $\ff$ and $r_N$, resp. $s_N$ its prefix, resp. non-prefix return word, then
$$\Psi(r_N)=\begin{pmatrix}
F_{N+1} \\ F_N
\end{pmatrix}; \ \Psi(s_N)=\begin{pmatrix}
F_{N} \\ F_{N-1}
\end{pmatrix}; \ \Psi(b_N)=\begin{pmatrix}
F_{N+2}-1 \\ F_{N+1}-1
\end{pmatrix}\,.$$
\end{proposition}

The previous proposition is borrowed from \citep{DvMePe20}, however, it can be proved directly using the morphism $\varphi$ generating the Fibonacci sequence and the simple relations   
$$ r_{N+1} =\varphi(r_N),  \quad s_{N+1} =\varphi(s_N) \    \ \text{and } \ \ 
b_{N+1} = \varphi(b_N){\tt a}\,.
$$

It  will be important to recognize whether a vector is the Parikh vector of a~certain factor of $\ff$.\begin{lemma}[{\citep{DDP21}}] \label{lem_kl}
Let $\begin{pmatrix} k \\ \ell \end{pmatrix} \in \mathbb N^2$. Then $\begin{pmatrix} k \\ \ell \end{pmatrix}$ is the Parikh vector of a factor of $\ff$ if and only if $|k-\tau \ell|<\tau^2$.
\end{lemma}
The derived sequence to each bispecial factor of the Fibonacci sequence $\ff$ is again $\ff$, as follows from~\citep{Araujo}.
The asymptotic critical exponent of $\mathbf f$ is $2+\tau$, see~\citep{Mignosi1992}. It is the minimum among Sturmian sequences~\citep{CaLu2000}.

\section{Colouring of the Fibonacci sequence}

Our aim is to find balanced sequences with a small asymptotic critical exponent.    
To find a suitable  sequence  over an alphabet of even size, we colour   the Fibonacci sequence $\ff$  by two periodic sequences.

Let $\delta \in \mathbb N, \delta \geq 1$. We construct a periodic sequence $\yy_\delta$ with the period $2^{\delta-1}$ recursively.  
We start with $\yy_1={\tt 1}^{\omega}$. Then we put on every second position the letter $\tt 2$ and obtain $\yy_2=(\tt 12)^{\omega}$. Further, we put on every second position in $\yy_2$ the letter $\tt 3$ and obtain $\yy_3=(\tt 1323)^{\omega}$, etc. Obviously, the sequence $\yy_\delta$ is  a {\em constant gap sequence} since for each letter in $\{{\tt 1,2},\dots, \delta\}$ the distance between consecutive occurrences of the letter is constant (as the reader can easily check). 
By $\hat{\yy}_\delta$  we denote a  sequence over $\{{\tt \hat{1},\hat{2}},\dots, \hat{\delta}\}$ defined in the same way as $\yy_\delta$. 
\begin{definition}\label{def:colouring}
Let $\delta \in \mathbb N, \delta \geq 1$. 
The {\em colouring} $\vv_\delta$ of the Fibonacci sequence $\ff$ by constant gap sequences $\yy_\delta$ and $\hat{\yy}_\delta$ 
is the sequence obtained from $\ff$ by replacing the subsequence of all ${\tt a}$'s with $\yy_\delta$ and the subsequence of all ${\tt b}$'s with~$\hat{\yy}_\delta$.
\end{definition}

\begin{example} \label{ex:colouring}
The sequence $\vv_3$, i.e., the colouring of $\ff$ by constant gap sequences $(\tt 1323)^{\omega}$ and $(\tt \hat{1}\hat{3}\hat{2}\hat{3})^{\omega}$, looks as follows:
\begin{align*}
\mathbf f &= \mathtt{abaababaabaababaababaabaababaab}\cdots \\
\vv_3 &= \mathtt{ \textcolor{red}{1}\hat{1}\textcolor{red}{32}\hat{3}\textcolor{red}{3}\hat{2}\textcolor{red}{13}\hat{3}\textcolor{red}{23}\hat{1}\textcolor{red}{1}\hat{3}\textcolor{red}{32}\hat{2}\textcolor{red}{3}\hat{3}\textcolor{red}{13}\hat{1}\textcolor{red}{23}\hat{3}\textcolor{red}{1}\hat{2}\textcolor{red}{32}\hat{3}} \cdots
\end{align*}
\end{example}

The following properties of  $\vv_\delta$ are direct consequences of general theorems on balanced sequences stated in ~\citep{Hubert00, Graham} and ~\citep{DDP21}. 
\begin{proposition} The sequence $\vv_\delta$ is aperiodic,  balanced and uniformly recurrent.    
\end{proposition}
We will also exploit the ``discolouration map'' $\pi$ which replaces in $\vv_\delta$  all letters from $\{{\tt 1,2},\dots, \delta\}$ by $\mathtt{a}$ and all letters from $\{{\tt \hat{1}, \hat{2}},\dots, \hat{\delta}\}$ by $\mathtt{b}$. Consequently, $\pi(\vv_\delta )=\ff$ and $\pi(v) \in \mathcal{L}(\ff)$ for every $v \in \LL(\vv_\delta )$.

\section{Return words in the colouring of the Fibonacci sequence}
We will treat the sequences $\vv_\delta$ for each $\delta \in \mathbb N, \delta \geq 1$, simultaneously. For this purpose we denote $H=2^{\delta-1}$ the period of the two  constant gap sequences we use to colour $\ff$. 

Let us recall that the length  $|v|$ of a return word to a factor $w$ in $\vv_\delta$ equals the difference between two consecutive occurrences of $w$ in $\vv_\delta$. Furthermore, if an index $i$ is an occurrence of  $w$ in $\vv_\delta$, then $i$ is an occurrence of its ``discolouration'' $\pi(w)$ in $\ff$. Hence,  $\pi(v)$  is a concatenation of several (at least one) return words to $\pi(w)$ in $\ff$. Thus if $r$ and $s$ are the return words to $\pi(w)$ in $\ff$, then  there exist $ k,\ell \in \mathbb{N}, k+\ell\geq 1$, such that 
\begin{equation}\label{kl}
   |v| = |\pi(v)|\ \ \text{and}\ \   \Psi(\pi(v)) = k\Psi(r) + \ell\Psi(s)\,. 
   \end{equation}
If $\pi(w)$ is a bispecial factor of $\ff$, then one of the return words to $\pi(w)$ is a prefix of $\ff$. In the sequel, we always assume that $r$ is the prefix return word to $\pi(w)$ and $s$ is the non-prefix one.

Of course,  not all pairs $(k,\ell)$  correspond to a concatenation of $r$ and $s$ forming  $\pi(v)$. The possible combinations are given by factors of the derived sequence of $\ff$ to  $\pi(w)$, which is however again the Fibonacci sequence. Formally, 
\begin{equation}\label{klParikh}
   \begin{pmatrix} k \\ \ell \end{pmatrix} \ \text{ is the Parikh vector of a factor in $\ff$}.
\end{equation}  

We plan to compute $E^*(\vv_\delta)$ by the formula from  Theorem \ref{thm:FormulaForCR}. Hence we are interested in bispecial factors 
and their return words in $\vv_\delta$. More precisely, we are interested only in long bispecial factors.    
The fact that both sequences $\yy_\delta$ and $\hat{\yy}_\delta$ are periodic with the same period $H$ implies some evident properties.

\begin{observation}\label{obs:H} Let $v$ be a return word in $\vv_{\delta}$ to a non-empty factor $w$ such that $\pi(w)$ satisfies both $|\pi(w)|_{\tt a}\geq H$ and $|\pi(w)|_{\tt b}\geq H$ (we call such factor $w$ sufficiently long).
\begin{enumerate}
\item $|\pi(v)|_{\tt a}$ and $|\pi(v)|_{\tt b}$ are divisible by  $H$.
\item If $w$ is a bispecial factor of $\vv_\delta$, then   $\pi(w)$ is a bispecial factor of $\ff$.  
 \end{enumerate}
\end{observation}
Item 2 of Observation~\ref{obs:H} enables us to exploit Proposition~\ref{prop:returnWords} on bispecial factors and return words in Sturmian sequences. 

\begin{lemma}\label{lem:kappa_lambda} 
Let $w$ be a sufficiently long bispecial factor of $\vv_\delta$ with $\delta \geq 1$.
Then for every return word $v$ to $w$ in $\vv_\delta$,  there exist $N, \kappa, \lambda \in \mathbb N, \kappa+\lambda\geq 1$, such that 
\begin{enumerate}
    \item $|w|=F_{N+3}-2$;
    \item $|v|=H(\kappa F_{N+2}+\lambda F_{N+1})$;
    \item  $|\kappa-\tau \lambda|<\frac{\tau^2}{H}$. 
\end{enumerate}
\end{lemma}
\begin{proof} By Item 2 of Observation \ref{obs:H}, $\pi(w)$ is bispecial in $\ff$, i.e., there exists $N$ such that  $\pi(w) = b_N=$ the $N^{th}$ bispecial factor of $\ff$.  As $|w| = |\pi(w)|$, 
Item 1 follows immediately from Proposition~\ref{prop:returnWords}.

We  show that $k, \ell$ from Equation \eqref{kl} are $H$-multiples of some integers.
For this purpose, we denote by  $A$  the  matrix from $ \mathbb{N}^{2\times 2}$ such that $ \Psi(r_N)$ and $\Psi(s_N)$ are the first and the second column of $A$, respectively. That is $A=\left(\begin{smallmatrix} F_{N+2} & F_{N+1}\\
F_{N+1} & F_{N}\end{smallmatrix}\right)$. Then by \eqref{kl} $\Psi(\pi(v))= A\left(\begin{smallmatrix} k \\ \ell \end{smallmatrix}\right)$. By Item 1 of Observation \ref{obs:H} 
$A\left(\begin{smallmatrix} k \\ \ell \end{smallmatrix}\right) =  H\left(\begin{smallmatrix} x \\ y \end{smallmatrix}\right)$
for some integers $x,y $. Since $A$ is unimodular, all entries of $A^{-1}$ are integers.  Hence $\left(\begin{smallmatrix} k \\ \ell \end{smallmatrix}\right) =  H A^{-1}\left(\begin{smallmatrix} x \\ y \end{smallmatrix}\right)\in H\mathbb{Z}^2$.   Therefore, $k=H\kappa$ and $\ell = H\lambda$ for some $\kappa, \lambda \in \N$. Moreover,   $\kappa+\lambda \geq 1 $, as $k$ and $\ell$ from \eqref{kl} satisfy $k +\ell \geq 1$.   Equality \eqref{kl} implies $|v| = k| r_N| + \ell |s_N|$. Substituting  $k=H\kappa$ and $\ell = H\lambda$ into this equality and using Proposition~\ref{prop:returnWords} we get Item 2 of the statement. By \eqref{klParikh}, $H\left(\begin{smallmatrix} \kappa \\ \lambda \end{smallmatrix}\right)$ is the Parikh vector of a factor of $\ff$. Lemma \ref{lem_kl} implies the inequality in Item 3.

\end{proof}

\section{An upper bound on asymptotic repetitive threshold of balanced sequences}
We know already the length of bispecial factors of $\vv_\delta$ and the form of their return words. In the sequel, we will estimate from above $E^*(\vv_\delta)$ applying Theorem~\ref{thm:FormulaForCR} and we will use this estimate to find an upper bound on the asymptotic repetitive threshold of balanced sequences. 
In order to estimate $E^*(\vv_\delta)$, the following lemma will play an essential role.

\begin{lemma}\label{lem:bounds}
Let $c>0$, $n \in \mathbb N, n \geq 1$, such that $$|F_{n+1}-\tau F_n| <c < |F_n -\tau F_{n-1}|\,.$$
If $|\kappa -\tau \lambda|<c$ for some $\kappa, \lambda \in \mathbb N, \kappa+\lambda\geq 1$, then
$\kappa \geq F_{n+1}$ and $\lambda \geq F_n$.
\end{lemma}
\begin{proof} Let us denote  $x:=F_{n}\kappa-F_{n+1}\lambda$  and $y:=F_{n-1}\kappa-F_{n} \lambda$. 
Using the first property of Fibonacci numbers from Section~\ref{sec:FibonacciNumbers}, we have
\begin{equation}\label{eq:star}
    \begin{array}{rcl}
  (-1)^n\kappa&=&(F_{n-1}\kappa-F_n\lambda)F_{n+1}+(-F_{n}\kappa+F_{n+1}\lambda)F_n = yF_{n+1}-xF_{n};  \\
  (-1)^n\lambda&=&(F_{n-1}\kappa-F_n\lambda)F_{n}+(-F_{n}\kappa+F_{n+1}\lambda)F_{n-1}= yF_{n}-xF_{n-1}\,.
    \end{array}
\end{equation}
Let us discuss several cases:
\begin{itemize}
    \item Let  $x=F_{n}\kappa-F_{n+1} \lambda=0$. Then $\frac{F_{n+1}}{F_n}=\frac{\kappa}{\lambda}$. Since $F_{n+1}$ and $F_n$ are coprime, $\kappa=m F_{n+1}$ and $\lambda =m F_n$ for some $m \in \mathbb N, m\geq 1$, thus the statement holds.

    \item 
Let  $y=F_{n-1}\kappa-F_{n} \lambda=0$. Then by the same reasoning we have $\kappa=m F_n$ and $\lambda=m F_{n-1}$ for some $m\in \mathbb N, m\geq 1$. If $m \geq 2$, then $\kappa \geq 2F_n \geq F_n+F_{n-1}=F_{n+1}$, and similarly, $\lambda \geq F_n$. If $m=1$, then $\kappa=F_n$, $\lambda=F_{n-1}$ and $|\kappa-\tau \lambda|>c$, i.e., the assumption of the lemma is not satisfied. 

\item Let  $x\cdot y < 0$. Then \eqref{eq:star} gives $\kappa =  |y|F_{n+1} + |x| F_n \geq F_{n+2}$ and $\lambda = |y|F_{n} + |x| F_{n-1} \geq F_{n+1} $.  

\item Let  $x\cdot y > 0$. Multiply the second equation in~\eqref{eq:star} by the number $-\tau$ and add it to the first equation to obtain 
$$|\kappa-\tau \lambda|=|y(F_{n+1}-\tau F_n)-x(F_{n}-\tau F_{n-1})|.$$
By   the fifth property of Fibonacci numbers from Section~\ref{sec:FibonacciNumbers},  $F_{n+1}-\tau F_n$ and   $F_{n}-\tau F_{n-1}$ have opposite signs. The condition  $x\cdot y > 0$ forces  
$$|\kappa-\tau \lambda|=|y|\cdot|F_{n+1}-\tau F_n|+|x|\cdot|F_{n}-\tau F_{n-1}| \geq |F_{n}-\tau F_{n-1}|  > c\,.$$
Therefore in this case the assumption of the lemma is not satisfied. 
\end{itemize}

\end{proof}

\begin{theorem}\label{thm:v_d}
Consider the balanced sequence $\vv_\delta$ with  $\delta \in \mathbb N, \delta\geq 1$.  Denote $H=2^{\delta-1}$. Find  $N_0\in \mathbb Z$ such that $\tau^{N_0+1}\leq H<\tau^{N_0+2}$.
Then $$E^*(\vv_\delta)\leq 1+\frac{1}{H \tau^{N_0-1}}.$$
\end{theorem}
\begin{proof}
If $\delta =1$, then $H =1$ and $N_0=-1$. The sequence $\vv_1$ is the Fibonacci sequence over the alphabet $\{\tt 1, \hat{1}\}$. The result of \citep{Mignosi1992} gives 
$$E^*(\vv_1)=2+\tau = 1+ \tau^2 = 1+ \frac{1}{H \tau^{N_0 -1}},$$ 
as we wanted to show. 

If $\delta =2$, then $H =2$ and $N_0=0$. The sequence $\vv_2$ is the sequence 
having the minimal critical exponent among all quaternary balanced sequences,  as shown in~\citep{RSV19}. The minimal value is  $E(\vv_2)=1+\frac{\tau}{2}$. Since the value $E(\vv_2)$  is irrational,  $$E^*(\vv_2)= E(\vv_2)=1+\frac{\tau}{2} = 1+ \frac{1}{H\tau^{N_0-1}}  .$$

Now we assume that $\delta\geq 3$.   Then $N_0\geq 1$. 
Let $w$ be a sufficiently long bispecial factor in $\vv_\delta$  and  $v$ its shortest return word.  By Observation \ref{obs:H},  $\pi(w) =b_N$ for some $N \in \mathbb N$.  
According to Lemma~\ref{lem:kappa_lambda}, there exist $\kappa, \lambda \in \mathbb N,  \kappa+\lambda\geq 1$, such that $$ |w| = F_{N+3} - 2, \qquad |v|=H(\kappa F_{n+2}+\lambda F_{n+1})\ \ \   \text{and} \ \ \ |\kappa-\tau \lambda|<\frac{\tau^2}{H}\,.$$
Using the third property of Fibonacci numbers from Section~\ref{sec:FibonacciNumbers} and the upper and lower bounds on $H$, we obtain for $N_0 \geq 1$
$$|F_{N_0+1}-\tau F_{N_0}|=\frac{1}{\tau^{N_0}}<\frac{\tau^2}{H} < \frac{1}{\tau^{N_0-1}}=|F_{N_0}-\tau F_{N_0-1}|.$$
Applying Lemma~\ref{lem:bounds} with $c =\frac{\tau^2}{H}$, we have $\kappa \geq F_{N_0+1}$ and $\lambda \geq F_{N_0}$.
Hence, 
$$\frac{|w|}{|v|}= \frac{F_{N+3}-2}{H(\kappa F_{N+2}+\lambda F_{N+1})}\leq  \frac{F_{N+3}-2}{H(F_{N_0+1} F_{N+2}+F_{N_0} F_{N+1})}\leq \frac{F_{N+3}}{HF_{N_0+N+2}} \to \frac{1}{H \tau^{N_0-1}}\,,$$
where we used the sixth and fourth property of Fibonacci numbers from Section~\ref{sec:FibonacciNumbers}.
Therefore by Theorem~\ref{thm:FormulaForCR}, we have for $N_0 \geq 1$
$$E^*(\vv_\delta) \leq 1+\frac{1}{H \tau^{N_0-1}}\,.$$
\end{proof}

\begin{theorem}\label{thm:upperbound}
Let $d\in \mathbb N, d>1$, $d$ even. Then $$1+\frac{1}{2^{d-2}} \leq RTB^*(d)< 1+\frac{\tau^3}{2^{d-2}} \,.$$
\end{theorem}
\begin{proof}
The lower bound was provided in~\citep{DOP2022}. 
Denote $\delta=d/2$. By Theorem~\ref{thm:v_d} we have $RTB^*(d)\leq E^*(\vv_\delta)\leq 1+\frac{1}{H \tau^{N_0-1}}$. 
Moreover, using the upper bound on $H$, we get 
$$\frac{1}{H\tau^{N_0-1}}< \frac{\tau^3}{H^2}=\frac{\tau^3}{2^{d-2}}\,,$$
which proves the upper bound from the theorem.
\end{proof}
\begin{corollary}\label{coro:liche}
Let $d\in \mathbb N, d>1$, $d$ even. Then $$RTB^*(d+1) \leq RTB^*(d)<1+\frac{\tau^3}{2^{d-2}} \,.$$
\end{corollary}
\begin{proof}
The upper bound on $RTB^*(d)$ was given in Theorem~\ref{thm:upperbound}. 
If we choose in any $d$-ary balanced sequence $\vv$ a letter and replace all its occurrences with the sequence $(\tt AB)^{\omega}$, then it is readily seen that we obtain a $(d+1)$-ary balanced sequence with the asymptotic critical exponent smaller than or equal to $E^*(\vv)$. Consequently, $RTB^*(d+1)\leq RTB^*(d)$.
\end{proof}

\section{Comments and open problems}
\begin{enumerate}
\item
  Theorem~\ref{thm:v_d} gives an upper bound on the asymptotic critical exponent of $\vv_\delta$.  In fact, we  are able to prove that $E^*(\vv_\delta)$ $=1+\frac{1}{H \tau^{N_0-1}}$ for $H$ and $N_0$ specified in the theorem. However, it demands some more theoretical background on balanced sequences and it was not necessary for our purposes, therefore we preferred to prove just the weaker version.
\item 
In Table~\ref{RTB*} taken from~\citep{DOP2022}, it is possible to see that the upper bound from Theorem~\ref{thm:v_d} is even optimal for $d=2,4,8$ and it is close to the value of $RTB^*(d)$ for every even $d \leq 10$.  We conjecture that the bound is attained  for infinitely many even $d$'s. We believe that it depends on  location of $H$ in the interval  $(\tau^{N_0+1}, \tau^{N_0+2})$ from   Theorem~\ref{thm:v_d}. 

\begin{table}[t]
\centering
\caption{The asymptotic repetitive threshold  $RTB^*(d)$ for alphabets of even size  $d\leq 10$ compared to the upper bound from Theorem~\ref{thm:v_d}. In the middle column we emphasize the cases where the equality holds. }\label{RTB*}
\setlength{\tabcolsep}{3.8pt}
\renewcommand{\arraystretch}{1.9}
\resizebox{0.9\textwidth}{!} {
\begin{tabular}{|r|l|c|l|}
\hline
$d$ &  $RTB^*(d)$&  & bound from Theorem~\ref{thm:v_d} 
\\
\hline
\hline
2  &$2 +\tau \doteq3.618034$ & $=$ & $2+\tau$ \\
\hline

4 & $1 + \frac{\tau}{2}\doteq 1.809017$ & $=$  & $1+\frac{\tau}{2}$ \\
\hline

6& $\frac{75+3\sqrt{65}}{80}\doteq 1.239835 $ & $<$  &
$1+\frac{1}{4}=1.25 $ \\

\hline
8& $1+\frac{1}{8\tau^2}\doteq 1.047746 $ & $=$ &$1+\frac{1}{8\tau^2}$  \\
\hline
10 & $\frac{364-21\sqrt{7}}{304}\doteq 1.0146027$  &  $<$ &
$1+\frac{1}{16\tau^3}\doteq 1.0147552$ \\
\hline
\end{tabular}
}
\end{table}

\item As a next step, it may be interesting to study the asymptotic repetitive threshold for $2$-balanced sequences. However, in contrast to balanced sequences, description of $2$-balanced sequences is missing. Sufficient conditions for 2-balancedness of Arnoux-Rauzy sequences were described by Berthé et al.~\citep{Cassaigne} and a class of 2-balanced $S$-adic sequences was provided by~Langeveld et al.~\citep{RoTh}. A question posed by Cassaigne at Journées Montoises in 2022 is whether the asymptotic repetitive threshold of $2$-balanced sequences is equal to one. 
\item In~\citep{CuMoRa2020}, the repetitive threshold was determined for a class of  binary sequences which are rich in palindromes.  To find the repetitive threshold and the asymptotic repetitive threshold for $d$-ary rich sequences seems to be a hard task, as again, description of rich sequences (even over binary alphabet) is not known.   

\end{enumerate}

\section*{Acknowledgements}

The authors  acknowledge financial support of M\v SMT by  founded project   {CZ.02.1.01/0.0/0.0/16\_019/0000778}.


\end{document}